\documentclass[a4paper]{amsart}%
\usepackage{amscd}
\usepackage{amsthm}
\usepackage{graphicx}
\usepackage{amsmath}
\usepackage{amsfonts}
\usepackage{amssymb}%
\providecommand{\U}[1]{\protect\rule{.1in}{.1in}}

\newtheorem{theorem}{Theorem}

\newtheorem{definition}[theorem]{Definition}

\newtheorem{lemma}[theorem]{Lemma}

\newtheorem{remark}[theorem]{Remark}

\begin{document}

\title[Fractional type integral operators on variable Hardy spaces]{Fractional type integral operators on variable Hardy spaces}
\author{Pablo Rocha and Marta Urciuolo \\ \textit{Version 11-06-2016}}
\address{Facultad de Matematica, Astronomia y Fisica, CIEM (Conicet), Universidad
Nacional de Cordoba, Ciudad Universitaria, 5000 Cordoba, Argentina}
\email{rp@famaf.unc.edu.ar and urciuolo@famaf.unc.edu.ar}
\thanks{\textbf{Key words and phrases}: Hardy spaces, Variable Exponents, Fractional Operators.}
\thanks{\textbf{2.010 Math. Subject Classification}: 42B25, 42B35.}
\thanks{Partially supported by Conicet and SecytUNC}
\maketitle

\begin{abstract}
Given certain $n\times n$ invertible matrices $A_{1},...,$\thinspace $A_{m}$
and $0\leq \alpha <n,$ in this paper we obtain the $H^{p(.)}\left( \mathbb{R}%
^{n}\right) \rightarrow L^{q(.)}\left( \mathbb{R}^{n}\right) $ boundedness
of the integral operator with kernel $k(x,y)=\left\vert x-A_{1}y\right\vert
^{-\alpha _{1}}...\left\vert x-A_{m}y\right\vert ^{-\alpha _{m}},$ where $%
\alpha _{1}+...+\alpha _{m}=n-\alpha $ and $p(.),$ $q(.)$ are exponent
functions satisfying log-H\"{o}lder continuity conditions locally and at
infinity related by $\frac{1}{q(.)}=\frac{1}{p(.)}-\frac{\alpha }{n}$. \ We
also obtain the $H^{p(.)}\left( \mathbb{R}^{n}\right) \rightarrow
H^{q(.)}\left( \mathbb{R}^{n}\right) $ boundedness of the Riesz potential
operator.
\end{abstract}

\section{\protect\bigskip Introduction}

Given a measurable function $p(.):\mathbb{R}^{n}\rightarrow \left( 0,\infty
\right) $ such that $0<\inf\limits_{x\in \mathbb{R}^{n}}p(x)\leq
\sup\limits_{x\in \mathbb{R}^{n}}p(x)<\infty ,$ let $L^{p(.)}\left( \mathbb{R%
}^{n}\right) $ denote the space of measurable functions such that for some $%
\lambda >0,$%
\[
\int \left\vert \frac{f(x)}{\lambda }\right\vert ^{p(x)}dx<\infty .
\]%
We set%
\[
\left\Vert f\right\Vert _{p(.)}=\inf \left\{ \lambda >0:\int \left\vert
\frac{f(x)}{\lambda }\right\vert ^{p(x)}dx\leq 1\right\} .
\]

We see that $\left( L^{p(.)}\left( \mathbb{R}^{n}\right) ,\left\Vert
f\right\Vert _{p(.)}\right) $ is a quasi normed space. As usual we will
denote $p_{+}=\sup\limits_{x\in \mathbb{R}^{n}}p(x)$ and $%
p_{-}=\inf\limits_{x\in \mathbb{R}^{n}}p(x).$

These spaces are referred to as the variable $L^{p}$ spaces. In the last
years many authors have extended tha machinery af classical harmonic
analysis to these spaces. See, for example \cite{C-F}, \cite{C-F-N}, \cite{D}, \cite{D-R}, \cite{K-R}.

In the famous paper \cite{F-S}, C. Fefferman and E. Stein defined the Hardy
space $H^{p}\left( \mathbb{R}^{n}\right) ,$ $0<p<\infty ,$ with the norn
given by
\[
\left\Vert f\right\Vert _{H^{p}}=\left\Vert
\sup\limits_{t>0}\sup\limits_{\varphi \in \mathcal{F}_{N}}\left\vert
t^{-n}\varphi (t^{-1}.)\ast f\right\vert \right\Vert _{p},
\]%
for a suitable family $\mathcal{F}_{N}.$ In the paper \cite{N-S}, E. Nakai
and Y. Sawano defined the Hardy spaces with variable exponents, replacing $%
L^{p}$ by $L^{p(.)}$ in the above norm and they investigate their several
properties.

Let $0\leq \alpha <n,$ let $p(.):\mathbb{R}^{n}\rightarrow \left( 0,\infty
\right) $ be a measurable function, and let $q(.)$ be defined by $\frac{1}{%
q(.)}=\frac{1}{p(.)}-\frac{\alpha }{n}.$ Given certain invertible matrices $%
A_{1},...A_{m},$ $m\geq 1,$ we study
\begin{equation}
T_{\alpha }f(x)=\int \left\vert x-A_{1}y\right\vert ^{-\alpha
_{1}}...\left\vert x-A_{m}y\right\vert ^{-\alpha _{m}}f(y)dy,  \label{T}
\end{equation}%
where $\alpha _{1}+...+\alpha _{m}=n-\alpha .$ We observe that in the case $%
\alpha >0,$ $m=1$ and $A_{1}=I,$ $T$ is the classical fractional integral
operator (also known as the Riesz potential) $I_{\alpha }.$

With respect to classical Lebesgue or Hardy spaces, in the case $m>1,$ in
the paper \cite{R-U}, we obtained the $H^{p}(\mathbb{R}^{n})-L^{q}(\mathbb{R}%
^{n})$ boundedness of these operators and we show that we cannot expect the $%
H^{p}(\mathbb{R}^{n})-H^{q}(\mathbb{R}^{n})$ boundedness of them.This is an
important difference with the case $m=1.$ Indeed, in the paper \cite{T-W},
M. Taibleson and G. Weiss, using the molecular characterization of the real
Hardy spaces, obtained the boundedness of $I_{\alpha }$ from $H^{p}(\mathbb{R%
}^{n})$ into $H^{q}(\mathbb{R}^{n}),$ $0<p\leq 1.$ In this paper we will
extend both results to the setting of variable exponents. Here and below we
shall postulate the following conditions on $p(.),$

\begin{equation}
\left\vert p(x)-p(y)\right\vert \leq \frac{c}{-\log \left\vert
x-y\right\vert },\text{ }\left\vert x-y\right\vert <\frac{1}{2},
\label{log local}
\end{equation}%
and\textit{\ }%
\begin{equation}
\left\vert p(x)-p(y)\right\vert \leq \frac{c}{\log \left( e+\left\vert
x\right\vert \right) },\text{ }\left\vert y\right\vert \geq \left\vert
x\right\vert .  \label{log fuera}
\end{equation}

We note that the condition (\ref{log fuera}) is equivalent to the existence
of constants $C_{\infty }$ and $p_{\infty }$ such that%
\begin{equation}
\left\vert p(x)-p_{\infty }\right\vert \leq \frac{C_{\infty }}{\log \left(
e+\left\vert x\right\vert \right) },\qquad x\in \mathbb{R}^{n}.
\label{log fuera p infinito}
\end{equation}

In Section 2 we recall the definition and atomic decomposition of the Hardy
spaces with variable exponents given in \cite{N-S}. We also state three
crucial lemmas, two of them refering to estimations of the $L^{p(.)}\left(
\mathbb{R}^{n}\right) $ norm of the characteristic functions of cubes and
the other one about the vector valued boundedness of the fractional maximal
operator.

In Section 3 we obtain the $H^{p(.)}(\mathbb{R}^{n})-L^{q(.)}(\mathbb{R}%
^{n}) $ boundedness of the operator $T_{\alpha }$ corresponding to the case $%
m>1,$ where $p(.)$ is an exponent function satisfying the log-H\"{o}lder
continuity conditions (\ref{log local}) and (\ref{log fuera p infinito}),
such that $p(A_{i}x)=p(x),$ $x\in \mathbb{R}^{n},$ $1\leq i\leq m,$ $%
0<p_{-}\leq p_{+}<\frac{n}{\alpha }$ and $\frac{1}{q(.)}=\frac{1}{p(.)}-%
\frac{\alpha }{n}.$

In Section 4 we get the $H^{p(.)}(\mathbb{R}^{n})-H^{q(.)}(\mathbb{R}^{n})$
boundedness of the Riesz potential $I_{\alpha }$ where $p(.)$ satisfies (\ref%
{log local}), (\ref{log fuera p infinito}), $0<p_{-}\leq p_{+}<\frac{n}{%
\alpha }$ and $\frac{1}{q(.)}=\frac{1}{p(.)}-\frac{\alpha }{n}.$

\textbf{Notation }The symbol $A\lesssim B$ stands for the inequality $A\leq
cB$ for some constant $c.$ The symbol $A\sim B$ stands for $B\lesssim
A\lesssim B.$ We denote by $Q\left( z,r\right) $ the cube centered at $%
z=(z_{1},...z_{n})$ with side lenght $r.$ Given a cube $Q=Q\left( z,r\right)
,~$we set $kQ=Q(z,kr)$ and $l\left( Q\right) =r.$ For a measurable subset $%
E\subseteq \mathbb{R}^{n}$ we denote by $\left\vert E\right\vert $ and $\chi
_{E}$ the Lebesgue measure of $E$ and the characteristic function of $E$
respectively. For a function $p(.):\mathbb{R}^{n}\rightarrow \left( 0,\infty
\right) $ we define $\underline{p}=\min \left( p_{-},1\right)$; also given a cube $Q$ define
$p_{-}(Q)= \inf \{p(x) : x \in Q\}$ and $p_{+}(Q)= \sup \{p(x) : x \in Q\}$. As usual we
denote with $S(\mathbb{R}^{n})$ the space of smooth and rapidly decreasing
functions and with $S^{\prime }(\mathbb{R}^{n})$ the dual space. If $\mathbf{%
\beta }$ is the multiindex $\mathbf{\beta =(\beta }_{1},...,\beta _{n})$
then $\left\vert \mathbf{\beta }\right\vert =\beta _{1}+...+\beta _{n}.$

\section{Preliminaries}

Given a measurable function $p(.):\mathbb{R}^{n}\rightarrow \left( 0,\infty
\right) $ such that $0<p_{-}\leq p_{+}<\infty ,$ in the paper \cite{N-S} E,
Nakai and Y. Sawano give a variety of distinct approaches, based on
differing definitions, all lead to the same notion of the Hardy space $%
H^{p(.)}.$

\begin{definition}
Define $\mathcal{F}_{N}=\left\{ \varphi \in S(\mathbb{R}^{n}):\sum\limits_{%
\left\vert \mathbf{\beta }\right\vert \leq N}\sup\limits_{x\in \mathbb{R}%
^{n}}\left( 1+\left\vert x\right\vert \right) ^{N}\left\vert \partial ^{%
\mathbf{\beta }}\varphi (x)\right\vert \leq 1\right\} $. Let $f\in S^{\prime
}(\mathbb{R}^{n}).$ Denote by $\mathcal{M}$ the grand maximal operator given
by
\[
\mathcal{M}f(x)=\sup\limits_{t>0}\sup\limits_{\varphi \in \mathcal{F}%
_{N}}\left\vert \left( t^{-n}\varphi (t^{-1}.)\ast f\right) \left( x\right)
\right\vert ,
\]%
where $N$ is a large and fixed integer. The variable Hardy space $%
H^{p(.)}\left( \mathbb{R}^{n}\right) $ is the set of all $f\in S^{\prime }(%
\mathbb{R}^{n})$ for which $\left\Vert \mathcal{M}f\right\Vert
_{p(.)}<\infty $. In this case we define $\left\Vert f\right\Vert
_{H^{p(.)}}=\left\Vert \mathcal{M}f\right\Vert _{p(.)}$.
\end{definition}

\begin{definition}
$((p(.),p_{0},d)-atom)$. Let $p(.):\mathbb{R}^{n}\rightarrow \left( 0,\infty
\right) $, $0<p_{-}\leq p_{+}<p_{0}\leq \infty $ and $p_{0}\geq 1.$ Fix an
integer $d\geq d_{p(.)}=\min \left\{ l\in \mathbb{N\cup }\left\{ 0\right\}
:p_{-}(n+l+1)>n\right\} .$ A function $a$ on $\mathbb{R}^{n}$ is called a $%
(p(.),p_{0},d)-$atom if there exists a cube $Q$ such that\newline
$a_{1})$ $supp\left( a\right) \subset Q,$\newline
$a_{2})$ $\left\Vert a\right\Vert _{p_{0}}\leq \frac{\left\vert Q\right\vert
^{\frac{1}{p_{0}}}}{\left\Vert \chi _{Q}\right\Vert _{p(.)}},$\newline
$a_{3})$ $\int a(x)x^{\alpha }dx=0$ for all $\left\vert \alpha \right\vert
\leq d.$
\end{definition}

\begin{definition}
For sequences of nonnegative numbers $\left\{ \lambda_{j}\right\} _{j=1}^{\infty }$
and cubes $\left\{ Q_{j}\right\} _{j=1}^{\infty }$ and for a function $p(.):%
\mathbb{R}^{n}\rightarrow \left( 0,\infty \right) $, we define
\[
\mathcal{A}\left( \left\{ \lambda_{j}\right\} _{j=1}^{\infty },\left\{
Q_{j}\right\} _{j=1}^{\infty },p(.)\right) =\left\Vert \left\{
\sum\limits_{j=1}^{\infty }\left( \frac{\lambda_{j}\chi _{Q_{j}}}{\left\Vert \chi
_{Q_{j}}\right\Vert _{p(.)}}\right) ^{\underline{p}}\right\} ^{\frac{1}{%
\underline{p}}}\right\Vert _{p(.)}.
\]%
\newline
The space $H_{atom}^{p(.),p_{0},d}\left( \mathbb{R}^{n}\right) $ is the set
of all distributions $f\in S^{\prime }(\mathbb{R}^{n})$ such that it can be
written as
\begin{equation}
f=\sum\limits_{j=1}^{\infty }\lambda_{j}a_{j}  \label{desc. atomica}
\end{equation}%
in $S^{\prime }(\mathbb{R}^{n}),$ where $\left\{ \lambda_{j}\right\}
_{j=1}^{\infty }$ is a sequence of non negative numbers, the $a_{j}`s$ are $%
(p(.),p_{0},d)-$atoms and $\mathcal{A}\left( \left\{ \lambda_{j}\right\}
_{j=1}^{\infty },\left\{ Q_{j}\right\} _{j=1}^{\infty },p(.)\right) <\infty
. $ One defines
\[
\left\Vert f\right\Vert _{H_{atom}^{p(.),p_{0},d}}=\inf \mathcal{A}\left(
\left\{ \lambda_{j}\right\} _{j=1}^{\infty },\left\{ Q_{j}\right\} _{j=1}^{\infty
},p(.)\right)
\]%
where the infimun is taken over all admissible expressions as in (\ref{desc.
atomica}).
\end{definition}

Theorem 4.6 in \cite{N-S} asserts that $\left\Vert f\right\Vert
_{H_{atom}^{p(.),p_{0},d}}\sim \left\Vert f\right\Vert _{H^{p(.)}},$ thus we
will study the behavior of the operators $T_{\alpha }$ on atoms.

The following lemmas are crucial to get the principal results.

\begin{lemma}
$($Lemma 2.2. in \cite{N-S}$)$ Suppose that $p(.)$ is a function satisfying (\ref%
{log local}), (\ref{log fuera p infinito}) and $0<p_{-}\leq p_{+}<\infty .$%
\newline
$1)$ For all cubes $Q=Q(z,r)$ with $z\in \mathbb{R}^{n}$ and $r\leq 1,$ we
have
\[
\left\vert Q\right\vert ^{\frac{1}{p_{-}(Q)}}\lesssim \left\vert
Q\right\vert ^{\frac{1}{p_{+}(Q)}}.
\]%
In particular, we have
\[
\left\vert Q\right\vert ^{\frac{1}{p_{-}(Q)}}\sim \left\vert Q\right\vert ^{%
\frac{1}{p_{+}(Q)}}\sim \left\vert Q\right\vert ^{\frac{1}{p(z)}}\sim
\left\Vert \chi _{Q}\right\Vert _{L^{p(.)}}.
\]%
$2)$ For all cubes $Q=Q(z,r)$ with $z\in \mathbb{R}^{n}$ and $r\geq 1,$ we
have
\[
\left\vert Q\right\vert ^{\frac{1}{p_{\infty }}}\sim \left\Vert \chi
_{Q}\right\Vert _{L^{p(.)}}.
\]%
Here the implicit constants in $\sim $ do not depend on $z$ and $r>0.$
\end{lemma}

\begin{lemma}
Let $A$ be an $n \times n$ invertible matrix. Let $p(.) : \mathbb{R}^{n} \rightarrow (0, \infty)$ be a function satisfying
(\ref{log local}), (\ref{log fuera p infinito}),
$0 < p_{-} \leq p_{+} < \infty$ and $p(Ax) = p(x)$ for all $x \in \mathbb{R}^{n}$. Given a sequence of cubes $Q_j = Q(z_j, r_j)$, we set $Q_j^{\ast} = Q(Az_j, 4 D r_j)$ for each $j \in \mathbb{N}$, where $D=\| A \|$. Then
$$\mathcal{A}\left( \left\{ \lambda_{j}\right\}
_{j=1}^{\infty },\left\{ Q_{j}^{\ast}\right\} _{j=1}^{\infty },p(.)\right) \lesssim
\mathcal{A}\left( \left\{ \lambda_{j}\right\}
_{j=1}^{\infty },\left\{ Q_{j}\right\} _{j=1}^{\infty },p(.)\right),$$
for all sequences of nonnegative numbers $\{ \lambda_j \}_{j=1}^{\infty}$ and cubes $\{ Q_j \}_{j=1}^{\infty}$.
\end{lemma}
\begin{proof}
Since $p(Ax)=p(x)$ for all $x \in \mathbb{R}^{n}$, a change of variable gives
$$\mathcal{A}\left( \left\{ \lambda_{j}\right\} _{j=1}^{\infty },\left\{
Q_{j}^{\ast}\right\} _{j=1}^{\infty },p(.)\right) =\left\Vert \left\{
\sum\limits_{j=1}^{\infty }\left( \frac{\lambda_{j}\chi _{Q_{j}^{\ast}}}{\left\Vert \chi
_{Q_{j}\ast}\right\Vert _{p(.)}}\right) ^{\underline{p}}\right\} ^{\frac{1}{%
\underline{p}}}\right\Vert _{p(.)}$$
$$= \left\Vert \left\{
\sum\limits_{j=1}^{\infty }\left( \frac{\lambda_{j}\chi _{A^{-1}Q_{j}^{\ast}}}{\left\Vert \chi
_{A^{-1}Q_{j}^{\ast}}\right\Vert _{p(.)}}\right) ^{\underline{p}}\right\} ^{\frac{1}{%
\underline{p}}}\right\Vert _{p(.)} = : \mathcal{A}\left( \left\{ \lambda_{j}\right\} _{j=1}^{\infty },\left\{
A^{-1}Q_{j}^{\ast}\right\} _{j=1}^{\infty },p(.)\right),$$
it is easy to check that $Q_j \subset A^{-1}Q_{j}^{\ast}$ for all $j$. Moreover, there exists a positive universal constant $c$ such that
$|A^{-1}Q_{j}^{\ast}| \leq c |Q_j|$ for all $j$. The same argument utilized in the proof of Lemma 4.8 in \cite{N-S} works in this case, so
$$\mathcal{A}\left( \left\{ \lambda_{j}\right\} _{j=1}^{\infty },\left\{
A^{-1}Q_{j}^{\ast}\right\} _{j=1}^{\infty },p(.)\right) \lesssim \mathcal{A}\left( \left\{ \lambda_{j}\right\} _{j=1}^{\infty },\left\{
Q_{j}\right\} _{j=1}^{\infty }, p(.)\right).$$
The proof is therefore concluded.
\end{proof}

Given $0<\alpha <n,$ we define the fractional maximal operator $M_{\alpha }$
by
\[
M_{\alpha }f(x)=\sup\limits_{Q}\frac{1}{\left\vert Q\right\vert ^{1-\frac{%
\alpha }{n}}}\int\limits_{Q}\left\vert f(y)\right\vert dy,
\]%
where $f$ is a locally integrable function and the supremum is taken over
all the cubes $Q$ which contain $x.$ In the case $\alpha =0,$ the fractional
maximal operator reduces to the Hardy-Littlewood maximal operator.

\begin{lemma}
Let $0\leq \alpha <n,$ let $p(.):\mathbb{R}^{n}\rightarrow \left( 1,\infty
\right) $ \ such that $p$ satisfies (\ref{log local}), (\ref{log fuera p
infinito}) and $1<p_{-}\leq p_{+}<\frac{n}{\alpha }.$ Then for $\theta \in
\left( 1,\infty \right) $ and $\frac{1}{q(.)}=\frac{1}{p(.)}-\frac{\alpha }{n%
}$ we have
\[
\left\Vert \left( \sum\limits_{j=1}^{\infty }\left( M_{\alpha }f_{j}\right)
^{\theta }\right) ^{\frac{1}{\theta }}\right\Vert _{q(.)}\lesssim \left\Vert
\left( \sum\limits_{j=1}^{\infty }\left\vert f_{j}\right\vert ^{\theta
}\right) ^{\frac{1}{\theta }}\right\Vert _{p(.)},
\]%
for all sequences of bounded measurable functions with compact support $\left\{ f_{j}\right\}
_{j=1}^{\infty }.$
\end{lemma}

\begin{proof}
For the case $0<\alpha <n,$ the boundedness of the fractional maximal
operator $M_{\alpha }$ from $L^{p}(w^{p})$ into $L^{q}(w^{q})$ for $1<p<%
\frac{n}{\alpha }$, $\frac{1}{q}=\frac{1}{p}-\frac{\alpha }{n}$ $\ $and for
all weights $w$ in the Muckenhoupt class $A_{p,q}$ (see \cite{M-W}) gives
the inequality 3.16, Theorem 3.23 in \cite{C-M-P}, for the pair $\left(
M_{\alpha }f,f\right) $, $f\in L^{p}\left( \mathbb{R}^{n}\right) .$ So for $%
1<\theta <\infty ,$%
\[
\left\Vert \left\{ \sum\limits_{j=1}^{\infty }\left( M_{\alpha }f_{j}\right)
^{\theta }\right\} ^{\frac{1}{\theta }}\right\Vert _{L^{q}\left(
w^{q}\right) }\lesssim \left\Vert \left( \sum\limits_{j=1}^{\infty
}\left\vert f_{j}\right\vert ^{\theta }\right) ^{\frac{1}{\theta }%
}\right\Vert _{L^{p}\left( w^{p}\right) }
\]%
for all weights $w$ in the Muckenhoupt class $A_{p,q}.$ Now $w\in A_{1}$
implies $w^{\frac{1}{q}}\in A_{p,q}$ so
\[
\left\Vert \left\{ \sum\limits_{j=1}^{\infty }\left( M_{\alpha }f_{j}\right)
^{\theta }\right\} ^{\frac{1}{\theta }}\right\Vert _{L^{q}\left( w\right)
}\lesssim \left\Vert \left( \sum\limits_{j=1}^{\infty }\left\vert
f_{j}\right\vert ^{\theta }\right) ^{\frac{1}{\theta }}\right\Vert
_{L^{p}\left( w^{\frac{p}{q}}\right) }
\]%
for all $w\in A_{1},$ thus the lemma follows , in this case, from Lemma
4.30 in \cite{C-M-P}. For the case $\alpha =0,$ Theorem 4.25 in \cite{C-M-P}
applies.
\end{proof}

\begin{lemma}
Let $0\leq \alpha <n,$ let $p(.):\mathbb{R}^{n}\rightarrow \left( 0,\infty
\right) $ such that $p$ satisfies (\ref{log local}), (\ref{log fuera p
infinito}) and $0<p_{-}\leq p_{+}<\frac{n}{\alpha }.$ If $\frac{1}{q(.)}=%
\frac{1}{p(.)}-\frac{\alpha }{n},$ then
\[
\mathcal{A}\left( \left\{ \lambda_{j}\right\} _{j=1}^{\infty },\left\{
Q_{j}\right\} _{j=1}^{\infty },q(.)\right) \lesssim \mathcal{A}\left(
\left\{ \lambda_{j}\right\} _{j=1}^{\infty },\left\{ Q_{j}\right\} _{j=1}^{\infty
},p(.)\right) ,
\]%
for all sequences of nonnegative numbers $\left\{ \lambda_{j}\right\}
_{j=1}^{\infty }$ and cubes $\left\{ Q_{j}\right\} _{j=1}^{\infty }.$
\end{lemma}

\begin{proof}
Since $l^{\underline{p}}\hookrightarrow l^{\underline{q}},$ we have
\begin{eqnarray*}
\mathcal{A}\left( \left\{ \lambda_{j}\right\} _{j=1}^{\infty },\left\{
Q_{j}\right\} _{j=1}^{\infty },q(.)\right) &=&\left\Vert \left\{
\sum\limits_{j=1}^{\infty }\left( \frac{\lambda_{j}\chi _{Q_{j}}}{\left\Vert \chi
_{Q_{j}}\right\Vert _{q(.)}}\right) ^{\underline{q}}\right\} ^{\frac{1}{%
\underline{q}}}\right\Vert _{q(.)} \\
&\lesssim &\left\Vert \left\{ \sum\limits_{j=1}^{\infty }\left( \frac{%
\lambda_{j}\chi _{Q_{j}}}{\left\Vert \chi _{Q_{j}}\right\Vert _{q(.)}}\right) ^{%
\underline{p}}\right\} ^{\frac{1}{\underline{p}}}\right\Vert _{q(.)}.
\end{eqnarray*}%
Now from Lemma 4 we obtain $\left\Vert \chi _{Q_{j}}\right\Vert _{q(.)}\sim
\left\Vert \chi _{Q_{j}}\right\Vert _{p(.)}\left\vert Q_{j}\right\vert ^{-%
\frac{\alpha }{n}}.$ Moreover a simple computation gives
\[
\left\vert Q_{j}\right\vert ^{\frac{\alpha }{n}}\chi _{Q_{j}}\left( x\right)
\leq M_{\frac{\alpha \underline{p}}{2}}(\chi _{Q_{j}})^{\frac{2}{\underline{p%
}}}\left( x\right) ,
\]%
so
\[
\lesssim \left\Vert \left\{ \sum\limits_{j=1}^{\infty }\left( \frac{%
\lambda_{j}\chi _{Q_{j}}\left\vert Q_{j}\right\vert ^{\frac{\alpha }{n}}}{%
\left\Vert \chi _{Q_{j}}\right\Vert _{p(.)}}\right) ^{\underline{p}}\right\}
^{\frac{1}{\underline{p}}}\right\Vert _{q(.)}\lesssim \left\Vert \left\{
\sum\limits_{j=1}^{\infty }\left( \frac{\lambda_{j}M_{\frac{\alpha \underline{p}}{2%
}}(\chi _{Q_{j}})^{\frac{2}{\underline{p}}}}{\left\Vert \chi
_{Q_{j}}\right\Vert _{p(.)}}\right) ^{\underline{p}}\right\} ^{\frac{1}{%
\underline{p}}}\right\Vert _{q(.)}
\]%
\[
=\left\Vert \left\{ \sum\limits_{j=1}^{\infty }\frac{\lambda_{j}^{\underline{p}}M_{%
\frac{\alpha \underline{p}}{2}}(\chi _{Q_{j}})^{2}}{\left\Vert \chi
_{Q_{j}}\right\Vert _{p(.)}^{\underline{p}}}\right\} ^{\frac{1}{\underline{p}%
}}\right\Vert _{q(.)}=\left\Vert \left\{ \sum\limits_{j=1}^{\infty }\frac{%
\lambda_{j}^{\underline{p}}M_{\frac{\alpha \underline{p}}{2}}(\chi _{Q_{j}})^{2}}{%
\left\Vert \chi _{Q_{j}}\right\Vert _{p(.)}^{\underline{p}}}\right\} ^{\frac{%
1}{2}}\right\Vert _{\frac{2q(.)}{\underline{p}}}^{\frac{2}{\underline{p}}}
\]
\begin{eqnarray*}
&\lesssim &\left\Vert \left\{ \sum\limits_{j=1}^{\infty }\frac{\lambda_{j}^{%
\underline{p}}\chi _{Q_{j}}}{\left\Vert \chi _{Q_{j}}\right\Vert _{p(.)}^{%
\underline{p}}}\right\} ^{\frac{1}{2}}\right\Vert _{\frac{2p(.)}{\underline{p%
}}}^{\frac{2}{\underline{p}}}=\left\Vert \left\{ \sum\limits_{j=1}^{\infty
}\left( \frac{\lambda_{j}\chi _{Q_{j}}}{\left\Vert \chi _{Q_{j}}\right\Vert _{p(.)}%
}\right) ^{\underline{p}}\right\} ^{\frac{1}{\underline{p}}}\right\Vert
_{p(.)} \\
&=&\mathcal{A}\left( \left\{ \lambda_{j}\right\} _{j=1}^{\infty },\left\{
Q_{j}\right\} _{j=1}^{\infty },p(.)\right) ,
\end{eqnarray*}%
where the third inequality follows from Lemma 6.
\end{proof}

\section{The main result}

\begin{theorem}
Let $m>1,$ let $A_{1},...,$\thinspace $A_{m}$ be $n\times n$ invertible
matrices such that $A_{i}-A_{j}$ is invertible for $i\neq j$, let $0\leq
\alpha <n$ and let $T_{\alpha }$ be the integral operator defined by (\ref{T}%
). Suppose $p(.):\mathbb{R}^{n}\rightarrow \left( 0,\infty \right) $
satisfies (\ref{log local}), (\ref{log fuera p infinito}), $0<p_{-}\leq
p_{+}<\frac{n}{\alpha }$ and $p(A_{i}x)\equiv p(x),$ $1\leq i\leq m.$ If
$\frac{1}{q(.)}=\frac{1}{p(.)}-\frac{\alpha }{n}$ then $T_{\alpha }$ can be extended to an $H^{p(.)}\left( \mathbb{R}^{n}\right) - L^{q(.)}\left(
\mathbb{R}^{n}\right)$ bounded operator.
\end{theorem}

\begin{proof}
Let $\max \{ 1 ,  p_{+} \} < p_0 < \frac{n}{\alpha}$. Given $f \in H^{p(.)} \cap L^{p_0}(\mathbb{R}^{n}),$ from Theorem 4.6 in \cite{N-S}
we have that there exist a sequence of nonnegative numbers $\{\lambda_j\}_{j=1}^{\infty}$, a sequence of cubes $Q_j = Q(z_j,r_j)$
centered at $z_j$ with side length $r_j$ and $(p(.), p_0, d)$ atoms $a_j$ supported on $Q_j$, satisfying
$$\mathcal{A}\left( \{\lambda_j\}_{j=1}^{\infty}, \{Q_j\}_{j=1}^{\infty}, p(.) \right) \leq c \|f \|_{H^{p(.)}},$$
such that $f$ can be decomposed as $f = \sum_{j \in \mathbb{N}} \lambda_j a_j$, where the convergence is in $H^{p(.)}$ and in $L^{p_0}$ (for the converge in $L^{p_0}$ see Theorem 5 in \cite{rocha}, the same argument works for $f \in H^{p(.)} \cap L^{p_0}(\mathbb{R}^{n})$).
We will study the behavior of $T_{\alpha}$ on atoms. Define $D=\max\limits_{1\leq i\leq m,\left\Vert
x\right\Vert \leq 1}\left\{ \left\Vert A_{i}(x)\right\Vert \right\} .$
Fix $j \in \mathbb{N}$, let $a_j$ be an $(p(.), p_0, d)$-atom supported on a cube $Q_j = Q(z_j, r_j)$, for each $1 \leq i \leq m$ let $Q_{ji}^{\ast} = Q(A_i z_j, 4Dr_j)$. Proposition 1 in \cite{R-U} gives that $T_{\alpha }$ is bounded from $L^{p_{0}}\left( \mathbb{R}^{n}\right) $ into
$L^{q_{0}}\left( \mathbb{R}^{n}\right)$ for $\frac{1}{q_{0}}=\frac{1}{p_{0}}-\frac{\alpha }{n},$ thus
\[
\left\Vert T_{\alpha} a_j \right\Vert _{L^{q_{0}}(Q_{j i}^{\ast })}\lesssim \left\Vert
a\right\Vert _{p_{0}}\lesssim \frac{\left\vert Q_j \right\vert ^{\frac{1}{p_{0}}
}}{\left\Vert \chi _{Q_j }\right\Vert _{p(.)}}\lesssim \frac{\left\vert
Q_{ji}^{\ast} \right\vert ^{\frac{1}{q_{0}}}}{\left\Vert \chi _{Q_{j i}^{\ast } }\right\Vert _{q(.)}},
\]
where the last inequality follows from lemma 4. So if $\left\Vert
T_{\alpha}a_j \right\Vert _{L^{q_{0}}(Q_{ji}^{\ast })}\neq 0$ we get
\begin{equation}
1\lesssim \frac{\left\vert Q_{ji}^{\ast} \right\vert ^{\frac{1}{q_{0}}}}{\left\Vert
T_{\alpha} a_j \right\Vert _{L^{q_{0}}(Q_{ji}^{\ast })}\left\Vert \chi _{Q_{j i}^{\ast }}\right\Vert
_{q(.)}}.  \label{1 menor o igual}
\end{equation}
We denote $k(x,y)=\left\vert x-A_{1}y\right\vert ^{-\alpha
_{1}}...\left\vert x-A_{m}y\right\vert ^{-\alpha _{m}}.$ In view of the moment condition of $a_j$ we have
\begin{equation}
T_{\alpha} a_j(x)=\int\limits_{Q_j}k(x,y)a_j(y)dy=\int\limits_{Q_j}\left( k(x,y)-q_{d, j}\left(
x,y\right) \right) a_j(y)dy,
\end{equation}%
\newline
where $q_{d, j}$ is the degree $d$ Taylor polynomial of the function $
y\rightarrow k(x,y)$ expanded around $z_j$. By the standard estimate of the
remainder term of the taylor expansion, there exists $\xi $ between $y$ and
$z_j$ such that
\[
\left\vert k(x,y)-q_{d, j}\left( x,y\right) \right\vert \lesssim \left\vert
y-z_j \right\vert ^{d+1}\sum\limits_{k_{1}+...+k_{n}=d+1}\left\vert \frac{%
\partial ^{d+1}}{\partial y_{1}^{k_{1}}...\partial y_{n}^{k_{n}}}k(x,\xi
)\right\vert
\]
\[
\lesssim \left\vert y-z_j \right\vert ^{d+1}\left(
\prod\limits_{i=1}^{m}\left\vert x-A_{i}\xi \right\vert ^{-\alpha
_{i}}\right) \left( \sum\limits_{l=1}^{m}\left\vert x-A_{l}\xi \right\vert
^{-1}\right) ^{d+1}.
\]
Now we decompose $\mathbb{R}^{n} = \bigcup_{i=1}^{m} Q_{ji}^{\ast} \cup R_j$, where $R_j = \left( \bigcup_{i=1}^{m} Q_{ji}^{\ast} \right)^{c}$, at the same time we decompose $R_j = \bigcup_{k=1}^{m} R_{jk}$ with
$$R_{jk} = \{ x \in R_{j} : |x - A_k z_j| \leq |x - A_i z_j| \,\, for \,\, all \,\, i \neq k \}.$$
If $x \in R_{j}$ then $|x - A_i z_j| \geq 2Dr_j$, since $\xi \in Q_j$ it follows that $|A_i z_j - A_i \xi | \leq D r_j \leq \frac{1}{2} |x - A_i z_j|$ so
$$|x - A_i \xi| = |x - A_i z_j + A_i z_j - A_i \xi| \geq |x - A_i z_j| - |A_i z_j - A_i \xi| \geq \frac{1}{2} |x - A_i z_j|.$$
If $x \in R_j$, then $x \in R_{jk}$ for some $k$ and since $\alpha_{1}+...+\alpha_{m} = n - \alpha$ we obtain
$$
\left\vert k(x,y)-q_{d, j}\left( x,y\right) \right\vert  \lesssim
\left\vert y-z_j \right\vert ^{d+1}\left( \prod\limits_{i=1}^{m}\left\vert
x-A_{i}z_j \right\vert ^{-\alpha _{i}}\right) \left(
\sum\limits_{l=1}^{m}\left\vert x-A_{l}z_j \right\vert ^{-1}\right) ^{d+1}
$$
$$
\lesssim r_{j}^{d+1}\left\vert x-A_{k}z_j \right\vert ^{-n+\alpha -d-1},
$$
this inequality allow us to conclude that
\begin{eqnarray*}
\left\vert T_{\alpha}a_j(x)\right\vert  &\lesssim &\left\Vert a_j \right\Vert
_{1}r_{j}^{d+1}\left\vert x-A_{k}z_j\right\vert ^{-n+\alpha -d-1} \\
&\lesssim &\left\vert Q_j \right\vert ^{1-\frac{1}{p_{0}}}\left\Vert
a_j \right\Vert _{p_{0}}r_{j}^{d+1}\left\vert x-A_{k}z_j\right\vert ^{-n+\alpha -d-1}
\\
&\lesssim &\frac{r_{j}^{n+d+1}}{\left\Vert \chi _{Q_j}\right\Vert _{p(.)}}%
\left\vert x-A_{k}z_j \right\vert ^{-n+\alpha -d-1} \\
&\lesssim &\frac{\left( M_{\frac{\alpha n}{n+d+1}}\left( \chi _{Q_j}\right)
(A_{k}^{-1}x)\right) ^{\frac{n+d+1}{n}}}{\left\Vert \chi _{Q_j}\right\Vert
_{p(.)}}, \,\,\,\,\, if \,\, x \in R_{jk}.
\end{eqnarray*}
Since $f = \sum_{j=1}^{\infty} \lambda_j a_j$ in $L^{p_0}$ and $T_{\alpha}$ is an $L^{p_0}-L^{\frac{n p_0}{n- \alpha p_0}}$ bounded operator, we have that $|T_{\alpha} f (x)| \leq \sum_{j=1}^{\infty} \lambda_j |T_{\alpha} a_j(x)|$. So
$$
|T_{\alpha} f (x)| \leq \sum_{j=1}^{\infty} \lambda_j |T_{\alpha} a_j(x)|
= \sum_{j=1}^{\infty} \left(\chi_{\bigcup_{i=1}^{m} Q_{ji}^{\ast}}(x) + \chi_{R_j}(x) \right) \lambda_j |T_{\alpha} a_j(x)|$$
$$\lesssim \sum_{j=1}^{\infty} \sum_{i=1}^{m} \chi_{Q_{ji}^{\ast}}(x) \lambda_j |T_{\alpha} a_j(x)| + \sum_{j=1}^{\infty} \sum_{k=1}^{m} \chi_{R_{jk}}(x) \lambda_j |T_{\alpha} a_j(x)|$$
$$\lesssim \sum_{j=1}^{\infty} \sum_{i=1}^{m} \chi_{Q_{ji}^{\ast}}(x) \lambda_j |T_{\alpha} a_j(x)|$$
$$+ \sum_{j=1}^{\infty} \sum_{k=1}^{m} \chi_{R_{jk}}(x) \lambda_j \frac{\left( M_{\frac{\alpha n}{n+d+1}}\left( \chi _{Q_j}\right)
(A_{k}^{-1}x)\right) ^{\frac{n+d+1}{n}}}{\left\Vert \chi _{Q_j}\right\Vert
_{p(.)}}= I + II$$
To study $I,$ if $\left\Vert T_{\alpha}a_{j}\right\Vert
_{L^{q_{0}}(Q_{ji}^{\ast })}\neq 0,$ we apply (\ref{1 menor o igual}) to
obtain, since $\underline{q}\leq 1,$
\[
\left\Vert I\right\Vert _{q(.)}\lesssim \sum\limits_{i=1}^{m}\left\Vert
\sum\limits_{j=1}^{\infty}\lambda_{j}\chi _{Q_{ji}^{\ast }}\left\vert T_{\alpha}
a_{j} \right\vert \right\Vert _{q(.)}\lesssim
\sum\limits_{i=1}^{m}\left\Vert \sum\limits_{j=1}^{\infty}\frac{\lambda_{j}\chi
_{Q_{ji}^{\ast }}\left\vert T_{\alpha} a_{j} \right\vert \left\vert
Q_{ji}^{\ast }\right\vert ^{\frac{1}{q_{0}}}}{\left\Vert
T_{\alpha}a_{j}\right\Vert _{L^{q_{0}}(Q_{ji}^{\ast })}\left\Vert \chi
_{Q_{ji}^{\ast }}\right\Vert _{q(.)}}\right\Vert _{q(.)}
\]
\[
\lesssim \sum\limits_{i=1}^{m}\left\Vert \left\{ \sum\limits_{j=1}^{\infty}\left(
\frac{\lambda_{j}\chi _{Q_{ji}^{\ast }}\left\vert T_{\alpha} a_{j}
\right\vert \left\vert Q_{ji}^{\ast }\right\vert ^{\frac{1}{q_{0}}}}{%
\left\Vert T_{\alpha}a_{j} \right\Vert _{L^{q_{0}}(Q_{ji}^{\ast })}\left\Vert \chi
_{Q_{ji}^{\ast }}\right\Vert _{q(.)}}\right) ^{\underline{q}}\right\} ^{%
\frac{1}{\underline{q}}}\right\Vert _{q(.)},
\]%
now we take $p_{0}$ near $\frac{n}{\alpha}$ such that $\delta =\frac{1}{q_{0}}$ satisfies
the hypothesis of Lemma 4.11 in \cite{N-S} to get
\[
\lesssim \sum_{i=1}^{m} \mathcal{A}\left( \left\{ \lambda_{j}\right\} _{j=1}^{\infty },\left\{
Q_{j i}^{\ast}\right\} _{j=1}^{\infty },q(.)\right)
\]
Lemma 5 gives
\[
\lesssim m \mathcal{A}\left( \left\{ \lambda_{j}\right\} _{j=1}^{\infty },\left\{
Q_{j}\right\} _{j=1}^{\infty },q(.)\right)
\]
now from Lemma 7,
\[
\lesssim \mathcal{A}\left( \left\{ \lambda_{j}\right\} _{j=1}^{\infty },\left\{
Q_{j}\right\} _{j=1}^{\infty },p(.)\right) \lesssim \left\Vert f\right\Vert
_{H^{p(.)}}.
\]%
\newline
To study $II,$ we observe that
\begin{eqnarray*}
\left\Vert II\right\Vert _{q(.)} = \left\Vert
\sum\limits_{j=1}^{\infty}\sum\limits_{k=1}^{m}\lambda_{j}\chi _{R_{jk}}\left( .\right)
\frac{\left( M_{\frac{\alpha n}{n+d+1}}\left( \chi _{Q_{j}}\right)
(A_{k}^{-1} .)\right) ^{\frac{n+d+1}{n}}}{\left\Vert \chi _{Q_j}\right\Vert
_{p(.)}}\right\Vert _{q(.)} \\
\lesssim m \left\Vert \left\{ \sum\limits_{j=1}^{\infty}\lambda_{j}%
\frac{\left( M_{\frac{\alpha n}{n+d+1}}\left( \chi _{Q_{j}}\right)
(A_{k}^{-1} .)\right) ^{\frac{n+d+1}{n}}}{\left\Vert \chi _{Q_{j}}\right\Vert
_{p(.)}}\right\} ^{\frac{n}{n+d+1}}\right\Vert _{\frac{n+d+1}{n}q(A_{k}.)}^{%
\frac{n+d+1}{n}}
\end{eqnarray*}%
\[
\lesssim \left\Vert \left\{ \sum\limits_{j=1}^{\infty}\lambda_{j}\frac{\chi _{Q_{j}}}{%
\left\Vert \chi _{Q_{j}}\right\Vert _{p(.)}}\right\} ^{\frac{n}{n+d+1}%
}\right\Vert _{\frac{n+d+1}{n}p(A_{k}.)}^{\frac{n+d+1}{n}}=\left\Vert
\sum\limits_{j=1}^{\infty}\lambda_{j}\frac{\chi _{Q_{j}}}{\left\Vert \chi
_{Q_{j}}\right\Vert _{p(.)}}\right\Vert _{p(.)}
\]%
\[
\lesssim \mathcal{A}\left( \left\{ \lambda_{j}\right\} _{j=1}^{\infty },\left\{
Q_{j}\right\} _{j=1}^{\infty },p(.)\right) \lesssim \left\Vert f\right\Vert
_{H^{p(.)}},
\]%
where the second inequality follows from Lemma 6, since $\frac{n+d+1}{n}
\underline{q}>1,$ the third inequality follows from Remark 4.4 in \cite{N-S}
and the second equality follows since $p(A_{k}x)\equiv p(x).$ Thus $$\| T_{\alpha} f \|_{q(.)} \lesssim \| f\|_{H^{p(.)}}$$
for all $f \in H^{p(.)} \cap L^{p_0}(\mathbb{R}^{n})$, so the theorem follows from the density of $H^{p(.)} \cap L^{p_0}(\mathbb{R}^{n})$ in $H^{p(.)}(\mathbb{R}^{n})$.
\end{proof}

\begin{remark} Observe that Theorem 8 still holds for $m=1$ and $0 < \alpha < n$. In particular, if $A_1=I$, then the Riesz potential is bounded from
$H^{p(.)}(\mathbb{R}^{n})$ into $L^{q(.)}(\mathbb{R}^{n})$.
\end{remark}

\begin{remark}

Suppose $h:\mathbb{R}\rightarrow \left( 0,\infty \right) $ that satisfies (%
\ref{log local}) and (\ref{log fuera p infinito}) on $\mathbb{R}$ and $%
0<h_{-}\leq h_{+}<\frac{n}{\alpha }.$ Let $p(x)=h(\left\vert x\right\vert )$
for $x\in \mathbb{R}^{n}$ and for $m>1$ let $A_{1},...,A_{m}$ be $n\times n$
orthogonal matrices such that $A_{i}-A_{j}$ is invertible for $i\neq j.$ It
is easy to check that (\ref{log local}),and (\ref{log fuera p infinito})
hold for $p$ and also that $0<p_{-}\leq p_{+}<\frac{n}{\alpha }$ and $%
p(A_{i}x)\equiv p(x),$ $1\leq i\leq m.$

Another non trivial example of exponent functions and invertible matrices
satisfying the hypothesis of the theorem is the following:

We consider $m=2,$
$p(.):\mathbb{R}^{n}\rightarrow \left( 0,\infty \right) $ that satisfies (%
\ref{log local}) and (\ref{log fuera p infinito}) , $0<p_{-}\leq p_{+}<\frac{%
n}{\alpha }$, and then we take $p_{e}(x)=p(x)+p(-x),$ $A_{1}=I$ and $%
A_{2}=-I.$
\end{remark}

\section{$H^{p(.)}\left( \mathbb{R}^{n}\right) -H^{q(.)}\left( \mathbb{R}%
^{n}\right) $ boundedness of the Riesz potential}

For $0<\alpha <n,$ let $I_{\alpha }$ be the fractional integral operator (or
Riesz potential) defined by
\begin{equation}
I_{\alpha }f(x)=\int_{\mathbb{R}^{n}} \frac{1}{\left\vert x-y\right\vert ^{n-\alpha }}f(y)dy,
\label{Ia}
\end{equation}%
$f\in L^{s}(\mathbb{R}^{n}),$ $1\leq s < \frac{n}{\alpha}$. A well known result of Sobolev gives the
boundedness of $I_{\alpha }$ from $L^{p}(\mathbb{R}^{n})$ into $L^{q}(%
\mathbb{R}^{n})$ for $1<p<\frac{n}{\alpha }$ and $\frac{1}{q}=\frac{1}{p}-%
\frac{\alpha }{n}.$ In \cite{C-F} C. Capone, D. Cruz Uribe and A. Fiorenza
extend this result to the case of Lebesgue spaces with variable exponents $%
L^{p(.)}.$ In \cite{S-W} E. Stein and G. Weiss used the theory of harmonic
functions of several variables to prove that these operators are bounded
from $H^{1}(\mathbb{R}^{n})$ into $L^{\frac{n}{n-\alpha }}(\mathbb{R}^{n}).$
In \cite{T-W}, M. Taibleson and G. Weiss obtained the boundedness of the
Riesz potential $I_{\alpha }$ from the Hardy spaces $H^{p}(\mathbb{R}^{n})$
into $H^{q}(\mathbb{R}^{n}),$ for $0<p<1$ and $\frac{1}{q}=\frac{1}{p}-\frac{%
\alpha }{n}.$ We extend these results to the context of Hardy spaces with
variable exponents. The main tools that we use are Lemma 7 and the molecular
decomposition developed in \cite{N-S}.

\begin{definition}
(Molecules) Let $0<p_{-}\leq p_{+}<p_{0}\leq \infty ,$ $p_{0}\geq 1$ and $%
d\in \mathbb{Z\cap }\left[ d_{p(.)},\infty \right) $ be fixed. One says that
$\mathfrak{M}$ is a $(p(.),p_{0},d)$ molecule centered at a cube $Q\,$%
centered at $z$ if it satisfies the following conditions.
\end{definition}

$1)$ On $2\sqrt{n}Q,$ $\mathfrak{M}$ satisfies the estimate $\left\Vert
\mathfrak{M}\right\Vert _{L^{p_{0}}(2\sqrt{n}Q)}\leq \frac{\left\vert
Q\right\vert ^{\frac{1}{p_{0}}}}{\left\Vert \chi _{Q}\right\Vert _{p(.)}}.$

$2)$ Outside $2\sqrt{n}Q,$ we have $\left\vert \mathfrak{M}\left( x\right)
\right\vert \leq \frac{1}{\left\Vert \chi _{Q}\right\Vert _{p(.)}}\left( 1+%
\frac{\left\vert x-z\right\vert }{l(Q)}\right) ^{-2n-2d-3}.$ This condition
is called the decay condition.

$3)$ If $\mathbf{\beta }$ is a multiindex with $\left\vert \mathbf{\beta }%
\right\vert \leq d,$ then we have
\[
\int\limits_{\mathbb{R}^{n}}x^{\mathbf{\beta }}\mathfrak{M}\left( x\right)
dx=0.
\]%
This condition is called the moment condition.

\begin{theorem}
Let $0<\alpha <n$ and let $I_{\alpha }$ be defined by (\ref{Ia}). If $p(.)$
is a measurable function that satisfies (\ref{log local}), (\ref{log fuera p
infinito}) and $0<p_{-}\leq p_{+}<\frac{n}{\alpha }$ and if $\frac{1}{q(.)}=\frac{1}{p(.)}-\frac{\alpha }{n}$,
then $I_{\alpha }$ can be extended to an $H^{p(.)}\left( \mathbb{R}^{n}\right) - H^{q(.)}\left(
\mathbb{R}^{n}\right)$ bounded operator.
\end{theorem}

\begin{proof}
Since $2d_{q(.)} + 2 + \alpha + n \geq d_{p(.)}$, as Theorem 8, given $p_0$ such that $max \{1, p_{+} \} < p_0 < \frac{n}{\alpha}$, we can decompose a distribution $f \in H^{p(.)}(\mathbb{R}^{n}) \cap L^{p_0}(\mathbb{R}^{n})$ as $f = \sum_{j=1}^{\infty} \lambda_j a_j$, where $a_j$ is an $(p(.), p_{0}, 2d_{q(.)} + 2 + \alpha + n)$-atom supported on the cube $Q_j$, where the convergence is in $H^{p(.)}(\mathbb{R}^{n})$ and in $L^{p_0}(\mathbb{R}^{n})$ so it is enough to show that if $a$ is an $(p(.), p_{0}, 2d_{q(.)} + 2 + \alpha + n)$-atom supported on the cube $Q$ centered at $z$, then $cI_{\alpha}(a)$ is a $(q(.), q_{0}, d_{q(.)})$-molecule centered at a cube $Q$ for some fixed constant $c > 0$ independent of the atom $a$. Indeed, since $f = \sum_{j=1}^{\infty} \lambda_j a_j$ in $L^{p_0}(\mathbb{R}^{n})$ then $I_{\alpha}f = \sum_{j=1}^{\infty} \lambda_{j} I_\alpha(a_{j})$ in $L^{\frac{n p_{0}}{n- \alpha p_0}}(\mathbb{R}^{n})$ and thus $I_{\alpha}f = \sum_{j=1}^{\infty} \lambda_{j} I_\alpha(a_{j})$ in $\mathcal{S}'$. Now from Theorem 5.2 in \cite{N-S} we obtain $$\| I_{\alpha}f \|_{H^{q(.)}} \lesssim \mathcal{A}(\{ \lambda_j \}, \{ Q_j\}, q(.))$$
by Lemma 7 and Theorem 4.6 in \cite{N-S} we have
$$ \lesssim \mathcal{A}(\{ \lambda_j \}, \{ Q_j\}, p(.)) \lesssim \| f \|_{H^{p(.)}},$$
for all $f \in H^{p(.)} \cap L^{p_0}(\mathbb{R}^{n})$, so the theorem follows from the density of $H^{p(.)} \cap L^{p_0}(\mathbb{R}^{n})$ in $H^{p(.)}(\mathbb{R}^{n})$.

Now we show that there exists $c>0$ such that $cI_{\alpha
}\left( a\right) $ is a $\left( q(.),q_{0},d_{q_{(.)}}\right) -$molecule%
\newline
$1)$ The Sobolev theorem and Lemma 4 give
\[
\left\Vert I_{\alpha }\left( a\right) \right\Vert _{L^{q_{0}}(2\sqrt{n}%
Q)}\lesssim \left\Vert a\right\Vert _{L^{p_{0}}}\leq \frac{\left\vert
Q\right\vert ^{\frac{1}{p_{0}}}}{\left\Vert \chi _{Q}\right\Vert _{p(.)}}%
\lesssim \frac{\left\vert Q\right\vert ^{\frac{1}{q_{0}}}}{\left\Vert \chi
_{Q}\right\Vert _{q(.)}}.
\]%
\newline
$2)$ Denote $d=2d_{q(.)}+2+\alpha +n.$ As in the proof of Theorem 8 we
get, for $x$ outside $2\sqrt{n}Q,$ that
\begin{eqnarray*}
\left\vert I_{\alpha }\left( a\right) (x)\right\vert  &\lesssim &\frac{%
l(Q)^{d+1}}{\left\vert x-z\right\vert ^{n-\alpha +d+1}}\left\Vert
a\right\Vert _{1}\lesssim \frac{1}{\left\Vert \chi _{Q}\right\Vert _{q(.)}}%
\left( \frac{l(Q)}{\left\vert x-z\right\vert }\right) ^{^{n-\alpha +d+1}} \\
&=&\frac{1}{\left\Vert \chi _{Q}\right\Vert _{q(.)}}\left( \frac{l(Q)}{%
\left\vert x-z\right\vert }\right) ^{^{2n+2d_{q(.)}+3}} \\
&\lesssim &\frac{1}{\left\Vert \chi _{Q}\right\Vert _{q(.)}}\left( \frac{l(Q)%
}{l(Q)+\left\vert x-z\right\vert }\right) ^{^{2n+2d_{q(.)}+3}} \\
&=&\frac{1}{\left\Vert \chi _{Q}\right\Vert _{q(.)}}\left( 1+\frac{%
\left\vert x-z\right\vert }{l(Q)}\right) ^{^{-2n-2d_{q(.)}-3}},
\end{eqnarray*}%
\newline
where the third inequality follows since $x$ outside $2\sqrt{n}Q$ implies $%
l(Q)+\left\vert x-z\right\vert <2\left\vert x-z\right\vert .$\newline
$3)$ The moment condition was proved by Taibleson and Weiss in \cite{T-W}.
\end{proof}

\end{document}